\documentclass[final,envcountsame,leqno]{svjour3}

\usepackage{amsmath,amssymb}
\usepackage[all]{xy}
\smartqed

\spnewtheorem*{iprop}{Proposition}{\bf}{\it}
\spnewtheorem*{ithm}{Theorem}{\bf}{\it}
\spnewtheorem*{icor}{Corollary}{\bf}{\it}

\DeclareMathOperator{\ann}{ann}

\DeclareMathOperator{\enm}{End}
\DeclareMathOperator{\ext}{Ext}
\DeclareMathOperator{\gr}{gr}
\DeclareMathOperator{\hb}{H}
\DeclareMathOperator{\hmm}{Hom}

\DeclareMathOperator*{\osum}{\oplus}

\DeclareMathOperator{\spec}{Spec}

\DeclareMathOperator{\diff}{Diff}
\DeclareMathOperator{\re}{Re}

\newcommand{\aff}{\mathbb A}

\newcommand{\cc}{\mathbb C}
\newcommand{\grdmod}{\mathbf{grMod}_{D}}
\newcommand{\grmod}[1]{\mathbf{grMod}_{#1}}
\newcommand{\grdhol}{\mathbf{grHol}_{D}}
\newcommand{\grsimp}{\mathbf{grSimp}}
\newcommand{\mcat}[1]{{}\mathbf{Mod}_{#1}}
\newcommand{\n}{\mathbb{N}}

\newcommand{\weyl}{A_1(k)}
\newcommand{\q}{\mathbb{Q}}
\newcommand{\z}{\mathbb{Z}}

\begin{document}

\title{Graded Holonomic D-modules on Monomial Curves}
\author{Eivind Eriksen}
\institute{BI Norwegian Business School, Department of Economics, N-0442
Oslo, Norway \\ \email{eivind.eriksen@bi.no}}
\date{\today}

\maketitle

\begin{abstract}
In this paper, we study the holonomic $D$-modules when $D$ is the ring of
$k$-linear differential operators on $A = k[\Gamma]$, the coordinate ring of
an affine monomial curve over the complex numbers $k = \cc$. In particular,
we consider the graded case, and classify the simple graded $D$-modules and
compute their extensions. The classification over the first Weyl algebra $D
= \weyl$ is obtained as a special case.
\end{abstract}

\section*{Introduction}

Let $\Gamma \subseteq \n_0$ be a numerical semigroup, generated by positive
integers $a_1, \dots, a_r$ without common factors, such that $H = \n_0
\setminus \Gamma$ is a finite set. We consider its semigroup algebra $A =
k[\Gamma]$ over the field $k = \cc$ of complex numbers. Since $X = \spec(A)$
is the affine monomial curve $X = \{ (t^{a_1}, t^{a_2}, \dots, t^{a_r}): t
\in k \} \subseteq \aff^r_k$, we call $A$ a \emph{monomial curve}.

We described the ring $D = \diff(A)$ of differential operators on a monomial
curve $A = k[\Gamma]$ in Eriksen \cite{erik03}, using the graded structure.
For any degree $w$, there is a homogeneous differential operator $P_w$ of
minimal order in $D_w$, given by
    \[ P_w = t^w \prod_{\gamma \in \Omega(w)} \, (E - \gamma) \]
where $E = t \partial$ is the Euler derivation, and $\Omega(w) = \{ \gamma
\in \Gamma: \gamma + w \not \in \Gamma \}$. Moreover, $E$ together with $\{
P_w: |w| \in \{ a_1, \dots, a_r \} \text{ or } |w| \in H \}$ is a set of
homogeneous generators for $D$, and $D_w = P_w \cdot k[E]$. We also showed
that the associated graded ring $\gr D$ is the finitely generated semigroup
algebra $\gr D = k[\Gamma'] \subseteq k[t,\xi]$, with
    \[ \Gamma' = \{ (m,n) \in \n_0^2: n \ge \sigma(m-n) \} \subseteq \n_0^2
    \]
where we write $\sigma(w)$ for the cardinality of $\Omega(w)$, which
satisfies $\sigma(w) = d(P_w)$.

In this paper, we use these results to study the holonomic $D$-modules on a
monomial curve. Note that $\gr D$ is a positively graded $k$-algebra via the
Bernstein filtration, but it is not generated by homogeneous elements of
degree one. We use the results of Shukla \cite{shuk80} to prove that for any
good filtration of $M$, there is a period $m \ge 1$ and polynomials $P_0,
\dots, P_{m-1} \in \q[t]$ such that $\dim_k \left( M_{nm+r} \right) = P_r(n)$
for $0 \le r < m$ and for all sufficiently large integers $n$. Moreover, the
polynomials
    \[ P_r(t) = \frac{e}{d!} \, t^d + \text{ terms of lower degree} \]
all have the same leading term. We define $d(M) = d$ and $e(M) = e$ to be the
dimension and multiplicity of $M$, and prove the Bernstein inequality $d(M)
\ge 1$.

We define a $D$-module to be holonomic if $d(M) = 1$, and use the properties
of dimension and multiplicity to study holonomic $D$-modules, generalizing
the standard results; see Coutinho \cite{cout95}. In particular, we show that
a $D$-module is holonomic if and only if it has finite length; this means
that we can study the length category of holonomic $D$-modules via its simple
modules and their extensions. The simplest example of a monomial curve is $A
= k[t]$, with the first Weyl algebra $D = \weyl$ as its ring of differential
operators. The simple modules over the first Weyl algebra were classified in
Block \cite{bloc81}, and form a very large set. We therefore restrict our
attention to the graded case.

The simple graded modules over the graded ring $D$ are the simple objects in
the length category $\grdhol$ of graded holonomic $D$-modules. We classify
the simple graded $D$-modules, up to graded isomorphisms and twists:

\begin{ithm}
If $A = k[\Gamma]$ is a monomial curve, then the simple graded left modules
over the ring $D$ of differential operators on $A$, up to graded isomorphisms
and twists, are
    \[ \{ M_0 \} \cup \{ M_{\alpha}: \alpha \in J^* \} \cup \{ M_{\infty} \}
    \]
where $J^* = \{ \alpha \in k: 0 \le \re(\alpha) < 1, \; \alpha \neq 0 \}$.
Moreover, we have that $M_0 = A$ and that $M_{\alpha} = D/D \cdot (E -
\alpha)$ for all $\alpha \in J^*$.
\end{ithm}

Any ring of differential operators on a monomial curve is Morita equivalent
to the first Weyl algebra $\weyl$, which is itself the ring of differential
operators on the monomial curve $A = k[t]$. We obtain the following special
case:

\begin{icor}
The simple graded left modules over the first Weyl algebra $D = \weyl$, up to
graded isomorphisms and twists, are given by $M_0 = D/D\partial$, $M_{\infty}
= D/Dt$, and $M_{\alpha} = D/D(E-\alpha)$ for all $\alpha \in J^*$.
\end{icor}

The Krull-Schmidt Theorem holds in $\grdhol$ since it is a length category.
In Eriksen \cite{erik18-itext}, we describe a constructive method for finding
the indecomposable objects in $\grdhol$ when the simple objects and their
extensions are given. We have therefore computed the extensions of the simple
graded $D$-modules:

\begin{iprop}
If $D$ is the ring of differential operators on a monomial curve, then the
extensions of $M_{\beta}$ by $M_{\alpha}$ in $\grdhol$ are given by
    \[ \ext^1_D(M_{\alpha},M_{\beta})_0 = \begin{cases} k \xi \cong k, &
    (\alpha,\beta) = (0,\infty), (\infty,0) \\ k \xi \cong k, & \alpha =
    \beta \in J^* \\ 0, & \text{ otherwise} \end{cases} \]
for all simple graded $D$-modules $M_{\alpha}, M_{\beta}$ with $\alpha,
\beta \in J^* \cup \{ 0, \infty \}$.
\end{iprop}

For example, when $D = \weyl$ is the first Weyl algebra, the proposition
shows that $\ext^1_D(M_{\infty}, M_0)_0 = k \xi \cong k$ for a non-split
graded extension $\xi$. We represent $\xi$ by the graded extension
    \[ 0 \to Dt/D\partial t \to D/D\partial t \to D/Dt \to 0 \]
where $D/Dt = M_{\infty}$ and $Dt/D\partial t \cong M_{0}[-1]$. Hence $M =
D/D \partial t$ is the unique indecomposable graded holonomic $D$-module of
length two with composition series $M \supseteq M_0[-1] \supseteq 0$ and
simple factors $M_{\infty}$ and $M_0[-1]$. 

It turns out that this example can be generalized. In fact, we use the 
results of this paper to classify all graded holonomic $D$-modules in 
Eriksen \cite{erik18-itext}, where we prove the following result: 

\begin{ithm}
Let $D = \weyl$ be the first Weyl algebra. The indecomposable $D$-modules
in $\grdhol$ are, up to graded isomorphisms and twists, given by
    \[ M(\alpha,n) = D/D \, (E-\alpha)^n, \quad M(\beta,n) = D/D \; \mathbf
    w(\beta,n) \]
where $n \ge 1$, $\alpha \in J^* = \{ \alpha \in k: 0 \le \re(\alpha) < 1,
\; \alpha \neq 0 \}$, $\beta \in \{ 0, \infty \}$, and $\mathbf
w(\beta,n)$ is the alternating word on $n$ letters in $t$ and $\partial$,
ending with $\partial$ if $\beta = 0$, and ending in $t$ if $\beta = \infty$.
\end{ithm}

We remark that the assumption that $k = \cc$ is one of convenience; the
results would hold over any algebraically closed field $k$ of characteristic
$0$, and the methods would be applicable if $k$ is any field of
characteristic $0$.

\section{Differential operators on affine monomial curves}
\label{s:diff-ring}

Let $\Gamma \subseteq \n_0$ be a numerical semigroup, and let $\{ a_1, a_2,
\dots, a_r \}$ be a minimal set of generators of $\Gamma$. Without loss of
generality, we may assume that these generators are without common factors,
such that $H = \n_0 \setminus \Gamma$ is a finite set. We consider the
semigroup algebra $A = k[\Gamma] \cong k[t^{a_1}, t^{a_2}, \dots, t^{a_r}]
\subseteq k[t]$ over the field $k = \cc$ of complex numbers, and the algebra
$D = \diff(A)$ of $k$-linear differential operators on $A$. Notice that $A =
k[\Gamma]$ is the coordinate ring of an affine monomial curve $X \subseteq
\aff^r$. By abuse of language, we shall call $A = k[\Gamma]$ a monomial
curve.

We notice that $A = k[\Gamma]$ is a positively graded ring. Let $S = \{ t^n:
n \in \Gamma \} \subseteq A$ be the multiplicatively closed subset consisting
of the homogeneous elements in $A$, and consider the graded localization $A
\subseteq S^{-1}A = k[t,t^{-1}] = T$. By general localization results for
rings of differential operators, there is a localization
    \[ D = \diff(A) \to S^{-1}D = S^{-1}A \otimes_A D = k[t,t^{-1}] \otimes_A
    D = \diff(T) \]
and we may identify $D$ with the subring $\{ P \in \diff(T): P(A) \subseteq A
\} \subseteq \diff(T)$; see for instance Smith, Stafford \cite{smit-staf88}.

Let us write $B = k[t]$ for the normalization of $A$, and $\diff(B)$ for its
algebra of $k$-linear differential operators. We remark that $B$ is itself
the coordinate ring of an affine monomial curve, the affine line $\aff^1$,
and that $\diff(B) = \weyl = k\langle t,\partial \rangle$ is the first Weyl
algebra, generated by $t$ and $\partial = d/dt$ and with relation $[\partial,
t] = 1$. In particular, it follows that
    \[ \diff(T) \cong k[t,t^{-1}] \otimes_{k[t]} \weyl \cong k[t,t^{-1}]
    \langle \partial \rangle \]
There is a grading on $\diff(T)$ induced by the grading on $T$, such that $P
\in \diff(T)$ is homogeneous of degree $w$ if and only if $P(T_i) \subseteq
T_{i+w}$ for all integers $i$. In concrete terms, $t^n \partial^m \in
\diff(T)$ is homogeneous of degree $n-m$. Moreover, $D = \diff(A)$ is the
graded subalgebra
    \[ D = \{ P \in \diff(T): P(A) \subseteq A \} \subseteq k[t,t^{-1}]
    \langle \partial \rangle \]
We write $D = \osum_w \, D_w$, where $D_w$ is the linear subspace of
differential operators of degree $w$.

Based on these results, we gave an explicit description of the algebra $D$ of
differential operators on an affine monomial curve in Eriksen \cite{erik03}:
For any degree $w \in \z$, we have that $D_w = P_w \cdot k[E]$, where $E = t
\partial$ is the Euler derivation of degree zero, and $P_w$ is the
homogeneous differential operator of degree $w$ given by
    \[ P_w = t^w \prod_{\gamma \in \Omega(w)} \; (E - \gamma) \]
determined by the set $\Omega(w) = \{ \gamma \in \Gamma: \gamma + w \not \in
\Gamma \}$. In particular, we proved the following result:

\begin{theorem}
If $A = k[\Gamma]$ is a monomial curve, then the $k$-algebra $D$ of
differential operators on $A$ is generated by the Euler derivation $E$ and
the differential operators $P_w$ for all degrees $w$ such that $|w| \in \{
a_1, a_2, \dots, a_r \}$ or $|w| \in H$.
\end{theorem}

The following structural results on the algebra $D$ of differential operators
on $A$ is a consequence of the work of Smith, Stafford \cite{smit-staf88}:

\begin{theorem} \label{t:diff-str}
If $A = k[\Gamma]$ is a monomial curve, then the ring $D$ of differential
operators on $A$ is Morita equivalent with the first Weyl algebra $\diff(B)
= \weyl$, and $D$ has the following properties:
\begin{enumerate}
\item $D$ is a simple Noetherian ring
\item $A$ is a simple left $D$-module
\item $D$ has Krull dimension $1$ and Gelfand-Kirillov dimension $2$
\item $D$ is a hereditary ring
\end{enumerate}
\end{theorem}
\begin{proof}
Since the normalization of $X = \spec(A)$ is the affine line $\spec(B) =
\aff^1$, and the normalization map $\aff^1 \to X$ is injective, it follows
from Proposition 3.3, Theorem 3.4, Theorem 3.7 and Proposition 4.2 in Smith,
Stafford \cite{smit-staf88} that $D$ is Morita equivalent with the Weyl
algebra $\weyl$, that $D$ is a simple ring, and that $A$ is a simple left
$D$-module. The rest follows from the fact that the properties are preserved
under Morita equivalence, and hold for $\weyl$. \qed
\end{proof}

Let $D^p \subseteq D$ be the set of differential operators in $D$ of order at
most $p$. We call $\{ D^p \}$ the order filtration of $D$, and consider the
associated graded ring
    \[ \gr D = \osum_{p \ge 0} D^p/D^{p-1} \]
It is well-known that $\gr D$ is a commutative ring, and that $\gr \diff(B) =
k[t,\xi]$, where $\xi$ is the image of $\partial$. In Eriksen \cite{erik03},
we proved that, for $\Gamma \neq \n_0$, the ring $\gr D = k[\Gamma'] \subseteq
k[t,\xi]$ is a semigroup algebra with minimal set of generators
    \[ \{ t^{\sigma(-w)} \xi^{\sigma(w)}: |w| \in \{ a_1, \dots, a_r \}
    \text{ or } |w| \in H \} \cup \{ t \xi \} \]
In the proof, we use the fact that $P_w$ has leading term $t^{w + \sigma(w)}
\partial^{\sigma(w)}$, where $\sigma(w)$ is the cardinality of $\Omega(w)$,
and that $\sigma(-w) = \sigma(w) + w$ for all integers $w$. In particular,
$\gr D = k[\Gamma'] $ is a finitely generated $k$-algebra of Krull dimension
two, and $\n_0^2 \setminus \Gamma'$ is a finite set.

Let us write $D^p_w = D^p \cap D_w \subseteq D$ for the linear subspace of
differential operators of degree $w$ and order at most $p$ in $D$, and define
    \[ B^n = \osum_{2p+w \le n} \; D^p_w \]
for all integers $n$. Then $\{ B^n \}$ is a filtration of $D$, which we call
the \emph{Bernstein filtration}. Note that any differential operator $P \in
B^n$ is of the form
    \[ P = \sum_{i+j \le n} \, c_{ij} \, t^i \partial^j \]
Clearly, the associated graded ring $\osum_n \, B^n/B^{n-1}$ is naturally
isomorphic to $\gr D$. Moreover, since $\gr D \subseteq k[t,\xi]$, it follows
that $B^0 = k$, that $B^n = 0$ for $n < 0$, and that $\dim_k B^n$ is finite
for all integers $n$.

\section{Holonomic D-modules on monomial curves}

Let $M$ be a left $D$-module. A \emph{filtration} of $M$ is a chain $M_0
\subseteq M_1 \subseteq \dots$ of $k$-linear subspaces of $M$ such that
$\cup_n \, M_n = M$ and $B^m \cdot M_n \subseteq M_{m+n}$ for all integers
$m,n \ge 0$. By convention, we let $M_n = 0$ for $n < 0$. We say that the
filtration is \emph{good} if the associated graded module $\gr M = \osum_n \,
M_n/M_{n-1}$ is a finitely generated module over $\gr D$.

We recall that there is a good filtration of $M$ if and only if $M$ is a
finitely generated $D$-module. Moreover, if $\{ M_n \}$ and $\{ M'_n \}$ are
two good filtrations of $M$, then there is an integer $N$ such that
    \[ M_{n-N} \subseteq M'_n \subseteq M_{n+N} \]
for all integers $n$. This is a standard result; see for instance Chapter 1
in Bj\"{o}rk \cite{bjor79}.

We fix a good filtration $\{ M_n \}$ of a left $D$-module $M$, and remark
that $\dim_k M_n$ is finite for all integers $n \ge 0$. In fact, we have that
$(\gr M)_n = M_n/M_{n-1}$ is finitely generated over $B_0 = k$, since $\gr M$
is a finitely generated module over $\gr D$, and $\gr D$ is positively graded
with $(\gr D)_0 = B_0 = k$. The Hilbert function of $\gr M$ is the function
$\hb(\gr M,-)$ given by
    \[ \hb(\gr M,n) = \dim_k \left( M_n/M_{n-1} \right) = \dim_k M_n - \dim_k
    M_{n-1} \]
and $\hb^1(\gr M,n) = \dim_k M_n$ is the first iterated Hilbert function of
$\gr M$.

\begin{proposition} \label{p:hilb-qp}
Let $M$ be a non-zero, finitely generated left $D$-module and let $\{ M_n \}$
be a good filtration of $M$. Then there is a positive integer $m \ge 1$ and
polynomials $P_r(t) \in \q[t]$ for $0 \le r < m$, such that $\dim_k \left(
M_{nm+r} \right) = P_r(n)$ for all sufficiently large integers $n$. Moreover,
we have that
    \[ P_r(t) = \frac{e}{d!} \cdot t^d + \text{ terms of lower degree} \]
for $0 \le r < d$, where $d$ is the Krull dimension of $\gr M$, and $e > 0$
is a positive integer.
\end{proposition}
\begin{proof}
To simplify notation, we write $R = \gr D$, and $N = \gr M$.  We consider the
Hilbert series $\hb_{N}(t) = \sum_n \hb(N,n) \cdot t^n$ of $N$ in $\z[[t]]$,
which can be written as a rational function
    \[ \hb_{N}(t) = \frac{Q(t)}{(1-t^l)^d} \]
where $Q(t) \in \z[t]$ with $Q(1) > 0$, and $l \ge 1$ is a positive integer.
This follows from Proposition 4.4.1 and the following remarks in Bruns,
Herzog \cite{brun-herz93}, since there is a homogeneous system of parameters
for $N$ of common degree $l > 0$. Let
    \[ m = \min \{ l \ge 1: \hb_{N}(t) \cdot (1-t^l)^d \in \z[t] \} \]
and let $a(t) = \hb_{N}(t) \cdot (1-t^m)^d \in \z[t]$. Then $a(1) > 0$, and
by Proposition 2.3 and Remark 2.4 in Campbell et al. \cite{cghsw00}, it
follows that there are polynomials $p_r(t) \in \q[t]$ for $0 \le r < m$ such
that $\deg p_r(t) \le d-1$ and $p_r(n)=\hb(N,nm+r)$ for all $n$ large enough.
Furthermore, since $a(1) > 0$, it follows that $p_r(t)$ has degree $d-1$ for
at least one integer $r$ with $0 \le r < m$. Let us write $a_r(t)$ for the
polynomial
    \[ a_r(t) = \sum_{n \ge 0} \, \alpha_{nm+r} \cdot t^n \]
in $\z[t]$, where the coefficients $\alpha_i$ for $i \ge 0$ are given by
$a(t) = \sum \, \alpha_i t^i$. We obtain polynomials $a_0(t), \dots, a_{m-1}
(t) \in \z[t]$ with $a(1) = a_0(1) + \dots + a_{m-1}(1)$. Furthermore, the
polynomial $p_r(t)$ has the form
    \[ p_r(t)= \frac{a_r(1)}{(d-1)!} \; t^{d-1} + \text{ terms of lower
    degree} \]
for $0 \le r < m$, where $a_r(1) \ge 0$ for all $r$. To compute the first
iterated Hilbert function, we choose a positive integer $n_0$ such that
$\hb(N,nm+r) = p_r(n)$ for $n > n_0$ and $0 \le r < m$. If we let $C_r =
\hb(N,nm+r+1) + \dots + \hb(N,nm+m-1)$, then we have, for $n > n_0$, that
\begin{align*}
\hb^1(N,nm + r) &= \sum_{s = 0}^{m-1} \sum_{t = 0}^{n} \hb(N,tm+s) - C_r \\
&= \sum_{s = 0}^{m-1} \sum_{t=n_0+1}^n p_s(t) - C_r + \hb^1(N,n_0 m + m-1) \\
&= \sum_{s = 0}^{m-1} \sum_{t = n_0+1}^n \frac{a_s(1)}{(d-1)!} \, t^{d-1} +
\text{ terms of lower degree } \\
&= \sum_{t = n_0 + 1}^{n} \frac{a(1)}{(d-1)!} \, t^{d-1} + \text{ terms of
lower degree } \\
&= \frac{a(1)}{d!} \, t^d + \text{ terms of lower degree }
\end{align*}
since $C_r$ and $\hb^1(N,n_0 m + m-1)$ are constants. Hence, there are
polynomials $P_r(t) \in \q[t]$ for $0 \le r < m$ of the form
    \[ P_r(t) = \frac{e}{d!} \, t^d + \text{ terms of lower degree } \]
such that $\hb^1(nm+r) = P_r(n)$ for all $n > n_0$, where $d$ is the Krull
dimension of $N$ and $e = a(1) > 0$ is a positive integer. \qed
\end{proof}

We remark that $d$ and $e$ in Proposition \ref{p:hilb-qp} are independent of
the chosen good filtration $\{ M_n \}$ of $M$. We define the \emph{dimension}
of $M$ to be $d(M) = d$, and the \emph{multiplicity} of $M$ to be $e(M) = e$
for any finitely generated left $D$-module $M \neq 0$. These invariants have
the following properties: If $0 \to M' \to M \to M'' \to 0$ is a short exact
sequence of $D$-modules, then $d(M) = \max \{ d(M'), d(M'') \}$. If moreover
$d(M') = d(M) = d(M'')$, then $e(M) = e(M') + e(M'')$. This follows from the
fact that a good filtration of $M$ induces good filtrations of $M'$ and $M''$.

\begin{proposition}[Bernstein's inequality]
For any finitely generated left $D$-module $M \neq 0$, we have $d(M) \ge 1$.
\end{proposition}
\begin{proof}
From Theorem \ref{t:diff-str}, it follows that $D$ is a simple ring, and
the algebra homomorphism  $D \to \enm_k(M)$ is therefore injective. This
implies that $M$ cannot be finite dimensional over $k$, hence $d(M) \ge 1$.
\qed
\end{proof}

We say that a finitely generated left $D$-module $M$ is \emph{holonomic} if
$M \neq 0$ and $d(M) = 1$. By convention, $M = 0$ is also considered to be
holonomic. If $M$ is holonomic, then it has finite length. In fact, if $M'
\subseteq M$ is a non-zero submodule of $M$, then $M'$ is holonomic and $e(M')
< e(M)$ if $M' \neq M$. Hence the length of $M$ is at most $e(M)$.

\begin{proposition} \label{p:dmod-hol}
Let $M$ be a finitely generated left $D$-module. Then the following
conditions are equivalent:
\begin{enumerate}
\item $M$ is holonomic
\item $M$ has finite length
\item $M$ is cyclic and not isomorphic to $D$
\end{enumerate}
\end{proposition}
\begin{proof}
If $M$ is holonomic, then it has finite length by the comment above. If $M$
has finite length, then $M$ is cyclic by Theorem 1.8.18 in Bj\"{o}rk
\cite{bjor79}, since $D$ is a simple ring by Theorem \ref{t:diff-str}.
Moreover, $M$ is not isomorphic to $D$ since $d(M) = 1$ and $d(D) = 2$. In
fact, the Bernstein filtration is a good filtration of $D$, and we have
\begin{align*}
\dim_k B^n &= 1 + 2 + \dots + (n+1) - s = \frac{(n+1)(n+2)}{2} - s \\
&= \frac12 n^2 + \text{ terms of lower degree}
\end{align*}
for $n$ sufficiently large, where $\gr D = k[\Gamma'] \subseteq k[t,\xi]$ and
$s$ is the cardinality of $\n_0^2 \setminus \Gamma'$. Finally, we show that
if $M$ is cyclic and not isomorphic to $D$, then it is holonomic. In this
case, we may assume that $M = D/I$ with $I \neq 0$, and there is a non-zero
element $P \in I$ and a principal left ideal $J = D \cdot P \subseteq I$. The
short exact sequence
    \[ 0 \to D \xrightarrow{\cdot P} D \to D/J \to 0 \]
shows that $d(D/J) = 1$. In fact, if $d(D/J) = 2$, then $e(D) = e(D) +
e(D/J)$, and this is a contradiction since $e(D) = 1$ from the computation
above. The short exact sequence $0 \to I/J \to D/J \to D/I \to 0$ gives that
$d(M) = d(D/I) \le d(D/J) = 1$, and this implies that $M$ is holonomic. \qed
\end{proof}

\section{Graded holonomic D-modules}

Let $D$ be the algebra of differential operators on a monomial curve $A =
k[\Gamma]$, with the $\z$-grading $D = \osum_w \, D_w$ described in Section
\ref{s:diff-ring}, and consider the category $\grdmod$ of $\z$-graded left
$D$-modules. An object of $\grdmod$ is a left $D$-module $M$ with a grading
    \[ M = \osum_w \, M_w \]
such that $D_v \cdot M_w \subseteq M_{v+w}$, and a morphism $\phi: M \to N$
in $\grdmod$ is a $D$-module homomorphism which is homogeneous of degree
zero, such that $\phi(M_w) \subseteq N_w$. For any graded $D$-module $M$ in
$\grdmod$ and any integer $n$, we denote by $M[n]$ the $n$'th twisted
$D$-module of $M$, with grading given by $M[n]_i = M_{n+i}$.

We wish to study and classify the graded holonomic $D$-modules, up to graded
isomorphisms in $\grdmod$ and twists. The full subcategory $\grdhol \subseteq
\grdmod$ of graded holonomic $D$-modules consists of graded $D$-modules $M$
of finite length, with composition series
    \[ M = M_0 \supseteq M_1 \supseteq \dots \supseteq M_{n-1} \supseteq M_n
    = 0 \]
in $\grdmod$ of length $n \le e(M)$; see Proposition \ref{p:dmod-hol}. We are
interested in the simple objects in $\grdmod$, since these are the simple
factors $M_i/M_{i+1}$ in the composition series.

\begin{lemma}
A graded left $D$-module $M$ is a simple object in $\grdmod$ if and only if
it is simple considered as a left $D$-module.
\end{lemma}
\begin{proof}
It is clear that if $M$ is simple as a left $D$-module, then it is simple
in $\grdmod$. To prove the converse, assume that $M$ is simple in $\grdmod$.
Then it follows from Theorem II.7.5 in N\u{a}st\u{a}sescu, van Oystaeyen
\cite{nast-oys79} that it is either simple or $1$-critical considered as a
left $D$-module. We claim that it cannot be $1$-critical. In fact, we may
choose a homogeneous element $m \neq 0$ of degree $w$ in $M$, and consider
the short exact sequence
    \[ 0 \to I \to D(-w) \xrightarrow{\cdot m} M \to 0 \]
where $D(-w)$ is a twist of $D$ and $I = \{ P \in D(-w): P \cdot m = 0 \}$
is the annihilator of $m$. Since $I \neq 0$, it follows from Proposition
\ref{p:dmod-hol} that $D/I$ is a holonomic $D$-module, and therefore of
finite length. This implies that $M$ is Artinian, of Krull dimension zero,
and it is therefore not $1$-critical. It follows that any simple object in
$\grdmod$ is simple considered as a left $D$-module. \qed
\end{proof}

The simple modules over the Weyl algebra $\weyl$ were classified in Block
\cite{bloc81}. We shall classify the simple objects in $\grdmod$ by adapting
Block's results to the graded situation. For any graded left $D$-module $M$,
we define
    \[ T_S(M) = \{ m \in M: sm = 0 \text{ for an element } s \in S \}
    \subseteq M \]
It follows from the fact that $S$ is an Ore set for $D$ that $T_S(M)$ is a
left $D$-module. We say that $M$ is an $S$-torsion module if $T_S(M) = M$,
and that it is an $S$-torsion free module if $T_S(M) = 0$. Notice that any
simple left $D$-module is either an $S$-torsion module or an $S$-torsion
free module, and that it is an $S$-torsion module if and only if $S^{-1}M =
0$.

There is a graded division algorithm in the ring $\diff(T)$ in the following
sense: For any homogeneous differential operators $P,Q \in \diff(T)$ with $Q
\neq 0$, we have that
    \[ P = L \cdot Q + R \]
for unique homogeneous differential operators $L,R \in \diff(T)$, where the
order $d(R) < d(Q)$. Therefore, any homogeneous left ideal $I \subseteq
\diff(T)$ is principal. In fact, if $P \neq 0$ is an element in $I$ with
minimal order, then $I = \diff(T) \cdot P$. This implies that a graded left
$\diff(T)$-module is simple considered as a $\diff(T)$-module if and only if
it is a simple object in the category of graded left $\diff(T)$-modules, and
we can talk about simple graded left $\diff(T)$-modules without ambiguity.

\begin{proposition} \label{p:grdmod-loc}
The assignment $M \mapsto S^{-1}M$ defines a bijective correspondence
    \[ S^{-1}: \grsimp_D[\text{S-torsion free}] \to \grsimp_{\diff(T)} \]
from the set of isomorphism classes of simple graded left $D$-modules that
are $S$-torsion free, and the set of isomorphism classes of simple graded
left $\diff(T)$-modules.
\end{proposition}
\begin{proof}
If $M$ is a simple graded left $D$-module that is $S$-torsion free, then
$S^{-1}M$ is a simple graded left $\diff(T)$-module, since $S$ consists of
homogeneous elements and localization is exact; see also Lemma 2.2.1 in
Block \cite{bloc81}. This defines the map $S^{-1}$, and that it is injective
follows from the proof of Lemma 2.2.1 in Block \cite{bloc81}. To show that
$S^{-1}$ is surjective, let us consider a graded simple left
$\diff(T)$-module $N$. We choose a non-zero homogeneous element $n \in N$
of degree $w$, which gives a short exact sequence
    \[ 0 \to I \to \diff(T)[-w] \xrightarrow{\cdot n} N \to 0 \]
where $I = \{ P \in \diff(T): P \cdot n = 0 \} = \ann(n)$. Let $J = I \cap
D[-w] \subseteq D[-w]$. This is a non-zero homogeneous ideal in $D[-w]$, and
$M = D[-w]/J \subseteq \diff(T)[-w]/I \cong N$ is a graded $D$-submodule.
Since $J \neq 0$, $M$ is a graded $D$-module of finite length by Proposition
\ref{p:dmod-hol}, hence it contains a graded simple left $D$-module $K$.
Then it follows from Lemma 2.2.1 in Block \cite{bloc81} and its proof that
$N \cong S^{-1}K$ since $K \subseteq M \subseteq N$. \qed
\end{proof}

Next, we classify the simple graded left $\diff(T)$-modules, up to graded
isomorphisms and twists. Any simple graded left $\diff(T)$-module $N$ must
be of the form $N \cong \diff(T)[-w]/I$, where $I = \diff(T) \cdot P$ is a
homogeneous principal left ideal that is maximal. Hence, we must have that
$P = t^w (E-\alpha)$ for $\alpha \in k$, and $t^w$ is a unit in $\diff(T)$.
We obtain the cases $\alpha = 0$, which gives the graded simple module
    \[ N_0 = \diff(T) / \diff(T) \cdot E = \diff(T) / \diff(T) \cdot \partial
    \cong T \]
and $\alpha \neq 0$, which gives the graded simple module
    \[ N_{\alpha} = \diff(T) / \diff(T) \cdot (E-\alpha) \]
It is not difficult to see that $N_{\alpha} \cong T[\alpha]$ if $\alpha \in
\z$, with isomorphism given by $1 \mapsto t^{\alpha}$. Similarly, we have
that $N_{\alpha} \cong N_{\beta}[\alpha-\beta]$ when $\alpha-\beta \in \z$,
with isomorphism given by $1 \mapsto t^{\alpha-\beta}$.

\begin{lemma}
Let $J = \{ \alpha \in k: 0 \le \re(\alpha) < 1 \}$. Then the set of simple
graded $\diff(T)$-modules, up to graded isomorphisms and twists, are given
by $\{ N_{\alpha}: \alpha \in J \}$.
\end{lemma}
\begin{proof}
We claim that if $\alpha - \beta \not \in \z$, then $N_{\alpha} \not \cong
N_{\beta}$ as left $\diff(T)$-modules. In light of the comments above, this
is enough to prove the lemma, since $J$ is a fundamental domain for $k/\z$.
To prove the claim, we consider the Weyl algebra $\weyl$ as a special case of
$D = \diff(A)$ with $A = k[\Gamma]$ and $\Gamma
= \n_0$. By Proposition \ref{p:grdmod-loc}, the left module $M_{\alpha} =
\weyl/\weyl \cdot (E-\alpha)$ corresponds to $N_{\alpha} = S^{-1}M_{\alpha}$
under localization when $\alpha \not \in \z$, and $M_0 = \weyl / \weyl \cdot
\partial$ corresponds to $N_0$. By results of Dixmier on the Weyl algebra, we
have that $M_{\alpha} \not \cong M_{\beta}$ when $\alpha -\beta \neq \z$; see
Lemma 24 in Dixmier \cite{dixm63} and Proposition 4.4 in Dixmier
\cite{dixm70}, and this proves the claim. \qed
\end{proof}

Let $A = k[\Gamma]$ be a monomial curve, let $D$ be the ring of differential
operators on $A$, and let $\weyl$ be the ring of differential operators on $B
= k[t]$. Then $D$ is Morita equivalent with the Weyl algebra $\weyl$, and by
Section 3.14 of Smith, Stafford \cite{smit-staf88}, the equivalence $\mcat D
\to \mcat{\weyl}$ of module categories is given by $M \mapsto D(A,B)
\otimes_D M$, where $D(A,B)$ is the $\weyl$-$D$ bimodule
    \[ D(A,B) = \{ P \in \diff(T): P(A) \subseteq B \} \subseteq \diff(T) \]
Notice that $D(A,B)$ is a graded $\weyl$-$D$ bimodule, with $S^{-1}D(A,B) =
\diff(T)$. Hence, there is an induced equivalence $F: \grdmod \to \grmod{
\weyl}$ of categories of graded modules, and it commutes with localization:
    \[ \xymatrix{ & \grmod{\diff(T)} & \\ \grdmod \ar[ur]^{S^{-1}}
    \ar[rr]_F & & \grmod{\weyl} \ar[ul]_{S^{-1}} } \]
We know that $\grsimp_D[\text{S-torsion free}] \cong \grsimp_{\diff(T)}$, and
we write $M_{\alpha}$ for the unique simple graded left $D$-module $M$ such
that $S^{-1}M = N_{\alpha}$ for $\alpha \in J$. In fact, it follows from the
proof of Proposition \ref{p:grdmod-loc} that $M_{\alpha}$ is the unique
simple submodule of $D/D \cdot (E - \alpha)$ if $\alpha \neq 0$, and that
$M_0$ is the unique simple submodule of $D/D \cdot \partial = A$. We claim
that $M_{\alpha} = A$ if $\alpha = 0$, and that $M_{\alpha} = D/D \cdot (E -
\alpha)$ if $\alpha \neq 0$. If $\alpha = 0$, this is clear since $A$ is a
simple left $D$-module. If $\alpha \neq 0$, then the claim follows from the
following result, since simple modules are preserved by Morita equivalence:

\begin{lemma}
If $A = k[t]$, then $M_{\alpha} = D/D \cdot (E - \alpha)$ for all $\alpha \in
J$ with $\alpha \neq 0$.
\end{lemma}
\begin{proof}
It is sufficient, in light of the comments above, to show that $D/D \cdot (E -
\alpha)$ is a simple module over $D = \weyl$. This follows from Lemma 24 in
Dixmier \cite{dixm63}. \qed
\end{proof}

Finally, we classify the simple graded $D$-modules that are $S$-torsion
modules. We first consider the case $A = k[t]$, and the $S'$-torsion modules
over the Weyl algebra $D = \weyl$, where $S' = k[t]^*$. In this case, it
follows from Proposition 4.1 and Corollary 4.1 in Block \cite{bloc81} that
the simple $S'$-torsion $D$-modules are given by
    \[ V(\beta) = D \otimes_A A/(t-\beta) \cong D / D \cdot (t - \beta) \]
for $\beta \in k$. Moreover, $V(\beta)$ is a graded left $D$-module, or an
$S$-torsion module, if and only if $\beta = 0$. Hence, $V(0) = D/D \cdot t$ is
the unique simple graded left module over the Weyl algebra $D = \weyl$ that is
$S$-torsion.

Let $D$ be the ring of differential operators on $A$ when $A = k[\Gamma]$ is
any monomial curve. Since $D$ is Morita equivalent to the Weyl algebra and the
property of being $S$-torsion is preserved under Morita equivalence, there is
a unique simple graded left $D$-module $M_{\infty}$ that is $S$-torsion.
Moreover, we have that
    \[ M_{\infty} = D(B,A) \otimes_{\weyl} V(0) \]
where $D(B,A) = \{ P \in \diff(T): P(B) \subseteq A \} \subseteq \diff(T)$ by
Section 3.14 in Smith, Stafford \cite{smit-staf88}. We summarize these
results as follows:

\begin{theorem}
If $A = k[\Gamma]$ is a monomial curve, then the simple graded left modules
over the ring $D$ of differential operators on $A$, up to graded isomorphisms
and twists, are given by
    \[ \{ M_0 \} \cup \{ M_{\alpha}: \alpha \in J^* \} \cup \{ M_{\infty} \}
    \]
where $J^* = \{ \alpha \in k: 0 \le \re(\alpha) < 1, \; \alpha \neq 0 \}$.
Moreover, we have that $M_0 = A$ and that $M_{\alpha} = D/D \cdot (E -
\alpha)$ for all $\alpha \in J^*$.
\end{theorem}

\begin{corollary}
The simple graded left modules over the first Weyl algebra $D = \weyl$, up to
graded isomorphisms and twists, are given by $M_0 = D/D\partial$, $M_{\infty}
= D/Dt$, and $M_{\alpha} = D/D(E-\alpha)$ for all $\alpha \in J^*$.
\end{corollary}

Let $M_{\alpha},M_{\beta}$ be simple graded $D$-modules in $\grdhol$, with
$\alpha, \beta \in J \cup \{ \infty \}$. Their extensions in $\grdhol$ are
given by $\ext^1_D(M_{\alpha},M_{\beta})_0$ and can be computed using free
resolutions. Moreover, extensions are preserved under Morita equivalence,
and it is therefore enough to consider the case $D = \weyl$.

\begin{proposition}
If $D$ is the ring of differential operators on a monomial curve, then the
extensions of $M_{\beta}$ by $M_{\alpha}$ in $\grdhol$ are given by
    \[ \ext^1_D(M_{\alpha},M_{\beta})_0 = \begin{cases} k \xi \cong k, &
    (\alpha,\beta) = (0,\infty), (\infty,0) \\ k \xi \cong k, & \alpha =
    \beta \in J^* \\ 0, & \text{ otherwise} \end{cases} \]
for all simple graded $D$-modules $M_{\alpha}, M_{\beta}$ with $\alpha,
\beta \in J^* \cup \{ 0, \infty \}$.
\end{proposition}
\begin{proof}
By the comments above, we may assume that $D = \weyl$ is the Weyl algebra.
We show the computation of the extensions in the case $\alpha = \infty$ and
$\beta = 0$; all other cases can be done in a similar way. The $D$-module
$M_{\infty} = D/Dt$ over the Weyl algebra $D = \weyl$ has the free resolution
    \[ 0 \gets M_{\infty} \gets D \xleftarrow{\cdot t} D \gets 0 \]
When we apply $\hmm_D(-,M_0)$ to this exact seqence, we obtain
    \[ \hmm_D(D,M_0) \cong M_0 \xrightarrow{t \cdot} \hmm_D(D,M_0) \cong M_0
    \to 0 \]
and $\ext^1_D(M_{\infty},M_0) \cong M_0/tM_0 \cong A/tA = k[t]/(t) \cong
k$, which is concentrated in degree zero. \qed
\end{proof}

\bibliographystyle{spmpsci}
\bibliography{eeriksen}

\begin{thebibliography}{10}
\providecommand{\url}[1]{{#1}}
\providecommand{\urlprefix}{URL }
\expandafter\ifx\csname urlstyle\endcsname\relax
  \providecommand{\doi}[1]{DOI~\discretionary{}{}{}#1}\else
  \providecommand{\doi}{DOI~\discretionary{}{}{}\begingroup
  \urlstyle{rm}\Url}\fi

\bibitem{bjor79}
Bj\"ork, J.E.: Rings of differential operators, \emph{North-Holland
  Mathematical Library}, vol.~21.
\newblock North-Holland Publishing Co., Amsterdam-New York (1979)

\bibitem{bloc81}
Block, R.E.: The irreducible representations of the {L}ie algebra {${\mathfrak
  s}{\mathfrak l}(2)$}\ and of the {W}eyl algebra.
\newblock Adv. in Math. \textbf{39}(1), 69--110 (1981).
\newblock \urlprefix\url{https://doi.org/10.1016/0001-8708(81)90058-X}

\bibitem{brun-herz93}
Bruns, W., Herzog, J.: Cohen-{M}acaulay rings, \emph{Cambridge Studies in
  Advanced Mathematics}, vol.~39.
\newblock Cambridge University Press, Cambridge (1993)

\bibitem{cghsw00}
Campbell, H.E.A., Geramita, A.V., Hughes, I.P., Smith, G.G., Wehlau, D.L.: Some
  remarks on {H}ilbert functions of {V}eronese algebras.
\newblock Comm. Algebra \textbf{28}(3), 1487--1496 (2000).
\newblock \urlprefix\url{https://doi.org/10.1080/00927870008826908}

\bibitem{cout95}
Coutinho, S.C.: A primer of algebraic {$D$}-modules, \emph{London Mathematical
  Society Student Texts}, vol.~33.
\newblock Cambridge University Press, Cambridge (1995).
\newblock \urlprefix\url{https://doi.org/10.1017/CBO9780511623653}

\bibitem{dixm63}
Dixmier, J.: Repr\'esentations irr\'eductibles des alg\`ebres de {L}ie
  nilpotentes.
\newblock An. Acad. Brasil. Ci. \textbf{35}, 491--519 (1963)

\bibitem{dixm70}
Dixmier, J.: Sur les alg\`ebres de {W}eyl. {II}.
\newblock Bull. Sci. Math. (2) \textbf{94}, 289--301 (1970)

\bibitem{erik03}
Eriksen, E.: Differential operators on monomial curves.
\newblock J. Algebra \textbf{264}(1), 186--198 (2003).
\newblock \urlprefix\url{https://doi.org/10.1016/S0021-8693(03)00144-3}

\bibitem{erik18-itext}
{Eriksen}, E.: {Iterated Extensions and Uniserial Length Categories}.
\newblock ArXiv e-prints  (2018)

\bibitem{nast-oys79}
N\u{a}st\u{a}sescu, C., Van~Oystaeyen, F.: Graded and filtered rings and
  modules, \emph{Lecture Notes in Mathematics}, vol. 758.
\newblock Springer, Berlin (1979)

\bibitem{shuk80}
Shukla, P.K.: On {H}ilbert functions of graded modules.
\newblock Math. Nachr. \textbf{96}, 301--309 (1980).
\newblock \urlprefix\url{https://doi.org/10.1002/mana.19800960123}

\bibitem{smit-staf88}
Smith, S.P., Stafford, J.T.: Differential operators on an affine curve.
\newblock Proc. London Math. Soc. (3) \textbf{56}(2), 229--259 (1988).
\newblock \urlprefix\url{https://doi.org/10.1112/plms/s3-56.2.229}

\end{thebibliography}

\end{document}